\documentclass{amsart}
\usepackage[all]{xy}
\usepackage{amsmath}
\usepackage{amsfonts}
\usepackage{amssymb}
\usepackage{graphics}
\usepackage{xy}
\usepackage{latexsym}

\date{}
\usepackage{enumerate}

\date{}

\newtheorem{theo}{Theorem}

\newtheorem{prop}[theo]{Proposition}

\newcommand{\beq}{\begin{equation}}
\newcommand{\eeq}{\end{equation}}

\begin{document}

\title{An elementary  proof of the halting property for chakravala algorithm}

\author{A. Bauval} 


\begin{abstract} In 1930, A. A. K. Ayyangar allegedly produced the missing proof that the ancient Indian chakravala algorithm -- designed to solve Pell's equation -- always halts. Refining his own elementary arguments, we give a correct and shorter proof.
\end{abstract}

\maketitle

\renewcommand{\thefootnote}{}
\renewcommand{\thefootnote}{\arabic{footnote}}
\setcounter{footnote}{0}

\section{Introduction}

The ancient empirical Indian `` cyclic algorithm'', to find a nontrivial solution of Pell's equation $x^2-ny^2=\pm1$ (where $n$ is some nonsquare positive integer) has long been considered as a small variant of the method later independently discovered by Europeans. As such, even renowned mathematicians credited Lagrange for the proof of its validity.\footnote{Ayyangar \cite{Ayyangar} and Selenius \cite{Selenius} give detailed historical analyses of many misconceptions about chakravala. See also \cite{sohum}.} In 1930, A. A. Krishnaswami Ayyangar \cite{Ayyangar} was the first to stress the originality of chakravala and (nearly) give the necessary proof that this more efficient algorithm also reaches the goal. His paper, though sometimes mentionned, does not seem to have been studied with much care,\footnote{Selenius \cite{Selenius} mentions six of Ayyangar's papers (perhaps including \cite{Ayyangar}, but with incomplete reference) but only as ``an attempt to "imitate" the {\sl chakrav\~ala\/} in the form of a continued fraction process'' and also writes: ``Though fairly reviewed, Ayyangar's work attracted very little attention, even in India.''}${}^,$\footnote{Edwards \cite{Edwards}, who does not mention Ayyangar, devotes more than one page (p. 35) of partial indications, through several exercises, to deduce the main properties of the Indian algorithm -- including what we call its halting property -- from those of the ``English'' one.} possibly due to the fact that his proof is rather lengthy. We give a correct and stronger version of his main theorem and use his own arguments to produce a much shorter proof of it.

It is a matter of taste to rephrase the study of both the Indian and European methods in terms of quadratic numbers, or of continued fractions and binary quadratic forms as Ayyangar did. We prefer to stick on using only elementary arithmetic on integers, thereby compromising the belief that the proof for chakravala is at least as hard as for the European algorithm, and was outside Bhaskara's reach (rather than just outside his experimental habits).

The paper is organized as follows: section \ref{best} introduces the two notions, ubiquitous in our paper, of ``best mod $k$ numbers'' -- integers whithin a congruence class which are best approximations of $\sqrt n$ in a certain sense -- and ``steps'', section \ref{description} presents chakravala algorithm --  roughly: a succession of steps -- section \ref{main} contains the main theorem -- according to which the algorithm is somehow reversible -- and sections \ref{twin} and \ref{halting without twins} use it repeatedly to explain why and how the process always halts.

\section{Best mod $k$ numbers, steps, reduced steps}\label{best}
We shall say that a positive integer $m$, chosen within a given congruence class mod $k$, is {\bf best mod $k$}, if $m^2$ is as near of $n$ as possible, i.e. for any positive $m'$ congruent to $m$ mod $k$, $|m^2-n|\le |m'^2-n|$. When  $k<\sqrt n$, such an $m$ must be one of the two elements $m_1,m_2$ of the class which are nearest to $\sqrt n$:
$$0<m_1<\sqrt n<m_2=m_1+k.$$
When only one of them is best, we shall call it {\bf strictly best}. If $k$ is even, it may happen that both are best, i.e. $n-m_1^2=m_2^2-n$. The following is a modified version of a property ``proved'' by Ayyangar\footnote{\cite{Ayyangar}, p. 237--238.}, in order to take this possibility into account.

\begin{prop}\label{propbest}
If some positive integers $k,k',m$ are such that
$$k<\sqrt n\quad{\rm and}\quad m^2-n=\varepsilon kk,\quad{\rm where}\quad\varepsilon=\pm1$$
then, the following are equivalent:
\begin{enumerate}[\upshape $(1)$]
\item $m$ is best mod $k$
\item $k'^2+k^2/4\le n$
\item $m\ge k'+\varepsilon k/2$
\end{enumerate}
and the inequalities in $(2)$ and $(3)$ are strict if and only if $m$ is strictly best.
\end{prop}

\begin{proof}If $m>\sqrt n$ then $\varepsilon=1$ and
$$(1)\Leftrightarrow m^2-n\le n-(m- k)^2\Leftrightarrow m^2-n\le km-k^2/2$$$$\Leftrightarrow kk'\le k(m-k/2)\Leftrightarrow(3)\Leftrightarrow m^2\ge k'^2+kk'+k^2/4\Leftrightarrow(2).$$
If $m<\sqrt n$ then  $\varepsilon=-1$ and
$$(1)\Leftrightarrow n-m^2\le(m+k)^2-n\Leftrightarrow n-m^2\le km+k^2/2,$$which is both equivalent to $kk'\le k(m+k/2)\Leftrightarrow(3)$ and to
$$ m^2+km-(n-k^2/2)\ge0\Leftrightarrow m\ge-k/2+\sqrt{n-k^2/4}$$$$\Leftrightarrow m^2\ge n-k\sqrt{n-k^2/4}\Leftrightarrow k\sqrt{n-k^2/4}\ge kk'\Leftrightarrow(2).$$
(This replaces Ayyangar's squaring argument\footnote{\cite{Ayyangar}, p. 237.}, which was not valid to prove $(3)\Rightarrow(2)$ in this case because $k'-k/2$ may be negative.)
\end{proof}
 
When these conditions are fulfilled (i.e. $k<\sqrt n$, $|m^2-n|=kk'$ and $m$ best mod $k$), we shall say that the triple $(k,m,k')$ is a {\bf step}. If $m$ is also best mod $k'$, we shall call the triple a {\bf reduced step} (this amounts to say that the reverse triple $(k',m,k)$ is also a step). By characterization $(2)$ of the proposition, any step $(k,m,k')$ satisfies $k'<\sqrt n$, and if $k'\ge k$, this step is reduced. 

\section{Chakravala algorithm}\label{description}
Given a nonsquare positive integer $n$, this algorithm produces four sequences of numbers $a_i,b_i,k_i,m_i$, by the following recipe:
\begin{itemize}
\item start the $0$-th stage with $m_{-1}=0,a_0=1,b_0=0,k_0=1$
\item for the $i$-th stage, select $m_i$ congruent to $-m_{i-1}$ mod $k_i$ and best\footnote{If there are two such $m_i$'s, no matter which one is chosen, the sequence of $k_i$'s and the solution eventually produced will be the same. This will be made clearer in section \ref{twin}.} mod $k_i$
\item set $a_{i+1}=(a_im_i+nb_i)/k_i$ and $b_{i+1}=(a_i+m_ib_i)/k_i$
\item set $k_{i+1}=|a_{i+1}^2-nb_{i+1}^2|$ (which is equal to $|m_i^2-n|/k_i$)
\item if $k_{i+1}=1$ then stop, else do the $i+1$-th stage.
\end{itemize}

An easy induction shows that
\begin{itemize}
\item $a_{i+1}$ and $b_{i+1}$ are integers, because for $i>0$, $|a_i(-m_{i-1})+nb_i|=k_ia_{i-1}$ and $|a_i+(-m_{i-1})b_i|=k_ib_{i-1}$,
\item they are coprime, because $|a_ib_{i+1}-b_ia_{i+1}|=1$,
\item $k_{i+1}<\sqrt n$, because $(k_i,m_i,k_{i+1})$ is a step.
\end{itemize}

\section{Main theorem}\label{main}
Whenever the algorithm halts, it produces a nontrivial solution of Pell's equation ($|a_{i+1}^2-nb_{i+1}^2|=k_{i+1}=1$ and $b_{i+1}>0$). Ayyangar noticed that the $a_i,b_i$'s may be forgotten in this halting problem, and claimed to prove the equivalent halting property for the algorithm below. In its formulation, we shall call {\bf successor} of a step $(k,m,k')$ the step (or one of the two steps) $(k',m',k'')$ such that $m'$ is congruent to $-m$ mod $k'$ and best mod $k'$, and $k''=|m'^2-n|/k'$:

\begin{itemize}
\item start the $0$-th stage with $m_{-1}=0,k_0=1$
\item at the $i$-th stage, take for $(k_i,m_i,k_{i+1})$ a successor of $(k_{i-1},m_{i-1},k_i)$ (only $m_{i-1}$ and $k_i$ are used for this)
\item if $k_{i+1}=1$ then stop, else do the $i+1$-th stage.
\end{itemize}

For instance if $n=m^2\pm1$, the sequence is reduced to a single step $(1,m,1)$ and produces the solution $m^2-n.1^2=\mp1$.

The heart of  Ayyangar's paper consists in ``proving'' that ``the'' successor of any {\sl reduced} step (produced or not by the algorithm) is also reduced.\footnote{This theorem is false with his definition of ``reduced'' -- corresponding to what we would call ``strictly reduced'' (meaning that $m$ is {\sl strictly} best mod $k$ and $k'$): we shall see in section \ref{twin} that a strictly reduced step may have two ``twin successors'', which are reduced, but of course not strictly.} A corollary is that every step of the sequence produced by the algorithm is reduced (since the $0$-th step $(1,m_0,k_1)$ is).  The same conclusion follows directly (without induction) from the following strengthening of his theorem:

\begin{theo}A successor of {\bf any} step is reduced, i.e. for any positive integers $k,k',k'',m,m'$ such that $k<\sqrt n$, $kk'=|m^2-n|$, $k'k''=|m'^2-n|$ and $k'\mid m+m'$, if
$$(1)\quad k'^2+\frac{k^2}4\le n\quad{\rm and}\quad(2)\quad k''^2+\frac{k'^2}4\le n,$$
then
$$(3)\quad k'^2+\frac{k''^2}4\le n.$$
\end{theo}

\begin{proof} Since $(3)$ follows from $(2)$ if $k'\le k''$ and from $(1)$ if $k''\le k$, assume from now on that $k<k''<k'$. Let$$\varepsilon=\frac{m^2-n}{kk'},\quad\varepsilon'=\frac{m'^2-n}{k'k''},\quad{\rm and}\quad l=\frac{m+m'}{k'},$$
then
$$k'l(m'-m)=m'^2-m^2=\varepsilon'k'k''-\varepsilon kk'.$$
Simplifying by $k'$ and combining with $m'+m=k'l$ leads to$$(4)\quad m'=\frac12\left(k'l+\frac{\varepsilon'k''-\varepsilon k}l\right)>\frac12\left(k'l+\frac{(\varepsilon'-1)k''}l\right)$$(because $k<k''$) and
$$m=\frac12\left(k'l-\frac{\varepsilon'k''-\varepsilon k}l\right).$$
From this expression of $m$ and hypothesis $(1)$, we deduce
$$\frac12\left(k'l-\frac{\varepsilon'k''-\varepsilon k}l\right)\ge k'+\varepsilon\frac k2$$
hence $l$ cannot be equal to $1$ because $k''<k'$, and when $\varepsilon'=1$, it cannot either be equal to $2$ because $k<k''$. This allows to eliminate $l$ from the lower bound $(4)$:
\begin{itemize}
\item if $\varepsilon'=-1$ then $l\ge 2$ and
$$m'>\frac12\left(k'l-\frac{2k''}l\right)\ge\frac12\left(2k'-\frac{2k''}2\right)=k'-\frac{k''}2$$
\item if $\varepsilon'=1$ then $l\ge 3$ and
$$m'>\frac{k'l}2\ge\frac{3k'}2>k'+\frac{k''}2.$$
\end{itemize}
Since $m'\ge k'+\varepsilon'\frac{k''}2$ is equivalent to $(3)$, this ends the proof.

\end{proof}

\section{Halting with twins}\label{twin}
Recall from section \ref{best} that for any step, there is either a unique `` strict'' successor, or a pair of what we shall call {\bf twin successors} $(k,m_\pm,k')$ with $k$ even, $m_\pm=k'\pm k/2$ and $n=k'^2+k^2/4$. This possibility was missed by Ayyangar, but we shall see that such a ``forking'' in the algorithm is local -- i.e. after the next step, the two variants of the sequence merge back to a single one -- and may occur only once. Moreover, such an ``accident'' will turn out to be more happy than troublesome.

\medskip\noindent{\bf Example.}~
For $n=29$, $5<\sqrt n<6$ and $n-5^2=4<7=6^2-n$ hence the first step $(1,m_0,k_1)$ is given by $m_0=5$ and $k_1=4/k_0=4$. Then, $m_1$ must be congruent to $-5$ mod $4$. Since $3<\sqrt n<7$ and $n-3^2=20=7^2-n$, the second step is a ``twin successor'': either $(4,3,5)$ or $(4,7,5)$. If we choose $(k_1,m_1,k_2)=(4,3,5)$ and compute the following steps, the whole sequence will be $(1,5,4),(4,3,5),(5,7,4),(4,5,1)$, whereas if we choose $(4,7,5)$, we obtain $(1,5,4),(4,7,5),(5,3,4),(4,5,1)$. Computing the solution associated to these two sequences gives the same result: $70^2-13^2.29=-1$.

\medskip\noindent{\bf General computations.}~
When
$$(k_i,m_i,k_{i+1}=k)\quad{\rm and}\quad(k_{i+1},m_{i+1},k_{i+2})=(k,m_\pm,k')$$
are two consecutive steps produced by the algorithm, let us find the two next steps.

By the main theorem, a successor of $(k,m_+,k')$ (resp. $(k,m_-,k')$)  is $(k',m_-,k)$ (resp. $(k',m_+,k)$) and it is the only one, otherwise $k'$ would be equal to $k$ and all the previous (and following) steps would be of the form $(k,m_\pm,k)$, which is impossible since $k$ is even, whereas $k_0=1$.

Similarly, a successor of $(k',m_-,k)$ (resp. $(k',m_+,k)$) is $(k,m_i,k_i)$ and it is the only one, otherwise $k_i$ would be equal to $k'$ and $m_i$ to $m_+$ or $m_-$, which is impossible. Indeed, all the previous (and following) $k_j$'s would then be, alternatively, equal to $k$ or $k'$, hence $k'$ would be equal to $1$ and $k$ to $2$ (since it divides $m_i+m_\pm$) but for $n=1^2+2^2/4=2$, the sequence is merely $(1,1,1)$ and has no twin.

Let us summarize these computations and draw a consequence:
\begin{prop}\label{proptwins}The algorithm produces only one or two sequences. In the latter case, the two sequences are finite, of the form $$(1,m_0,k_1),(k_1,m_1,k_2)\ldots,(k_{i-1},m_{i-1},k_i)(k_i,m_i,k),(k,k'+k/2,k'),$$
$$(k',k'-k/2,k),(k,m_i,k_i),(k_i,m_{i-1},k_{i-1})\ldots,(k_2,m_1,k_1),(k_1,m_0,1),$$
and its reverse,
$$(1,m_0,k_1),(k_1,m_1,k_2)\ldots,(k_{i-1},m_{i-1},k_i)(k_i,m_i,k),(k,k'-k/2,k'),$$
$$(k',k'+k/2,k),(k,m_i,k_i),(k_i,m_{i-1},k_{i-1})\ldots,(k_2,m_1,k_1),(k_1,m_0,1),$$
\end{prop}

\begin{proof}Let us keep the notations of the previous computations to denote the first ``twin step'', if any.  By the main theorem, the two sequences of the proposition -- let us call them $s$ and $s'$ -- are produced by the algorithm. We shall show that they are the only ones. Let $s''$ be any other one. By the general computations above, $s''$ differs from $s$ only by extremely local forkings; in particular, it is finite and has the same length. By the main theorem, its reverse is also produced by the algorithm. By minimality of $i$, $s''$ and its reverse therefore coincide with $s$ up to the $i$-th step, hence $s''$ is equal to $s$ or $s'$.
\end{proof}

Moreover, an easy calculation shows that these two sequences produce the same $(a_j,b_j)$'s -- except the middle one -- hence the same solution of Pell's equation.

\section{Halting without twins}\label{halting without twins}

\begin{prop}When the sequence produced by the algorithm is unique, it is finite and of the form
$$(1,m_0,k_1),(k_1,m_1,k_2)\ldots,(k_2,m_1,k_1),(k_1,m_0,1).$$
\end{prop}

\begin{proof}By hypothesis, there is no twin successor in the sequence, hence there is no ``twin predecessor'' either, by reversal in the general computations of the previous section. By proposition \ref{propbest}, the set of possible steps is finite, hence if the sequence was infinite, the $(i+p)$-th step would be equal to the $i$-th step, for some $i\ge 0$ and $p>0$, so that (since there are no ``twin predecessors'') the $p$-th step would be equal to the $0$-th, $k_p$ would be $1$, and there would be no $p$-th step at all (the algorithm would stop at the end of the $(p-1)$-th stage). This contradiction ends the proof that the sequence is finite. Again, by the main theorem, the reverse sequence is also produced by the algorithm hence (by uniqueness) both are equal.
\end{proof}

\noindent{\bf Remark.} In the previous section, we saw that as soon as some twin step is met, the middle of the sequence is reached and the remaining $k_j,m_j$'s are known, hence only the remaining $a_j,b_j$'s need further computation. By the main theorem, the same happens as soon as we meet some step of the form $(k,m,k)$ or some pair of consecutive steps of the form $(k,m,k'),(k',m,k)$. Therefore, the last proposition contains the main result, from a pragmatic point of vue: either some twin step is met, or one of these two configurations.

\author{Anne Bauval}\\
\address{\small Institut de Math\'ematiques de Toulouse\\
\' Equipe \' Emile Picard, UMR 5580\\
Universit\'e Toulouse III\\
118 Route de Narbonne, 31400 Toulouse - France\\
e-mail: bauval@math.univ-toulouse.fr}


\begin{thebibliography}{999}
\bibitem{Ayyangar} A. A. Krishnaswami Ayyangar, ``New light on Bhaskara's Chakravala or cyclic method of solving indeterminate equations of the second degree in two variables", {\sl J. Indian Math. Soc.} {\bf 18} (1929-30), 225--248.
\bibitem{sohum} \emph{Historical work of K. Ayyangar}, collected and introduced by his son A. K. Srinivasan,  http://www.ms.uky.edu/~sohum/AAK/PRELUDE.htm.
\bibitem{Edwards} Harold M. Edwards, \emph{Fermat's Last Theorem}, GTM {\bf 50}, Springer, 1977.
\bibitem{Selenius} Clas-Olof Selenius,  ``Rationale of the chakrav\~ala process of Jayadeva and Bh\~askara II", {\sl Historia Mathematica} {\bf 2} (1975), 167--184.
\end{thebibliography}
\end{document}